\documentclass{amsart}
\usepackage{amsfonts}
\usepackage{graphicx}
\usepackage{amscd}
\usepackage{amsmath}
\usepackage{amssymb}
\usepackage{stmaryrd} 

\makeatletter
\@namedef{subjclassname@2010}{%
  \textup{2010} Mathematics Subject Classification}
\makeatother

\theoremstyle{plain}
\newtheorem*{acknowledgement}{Acknowledgement}

\newtheorem{corollary}{\bf Corollary}
\newtheorem{definition}{\bf Definition}
\newtheorem{lemma}{\bf Lemma}

\newtheorem{proposition}{\bf Proposition}
\newtheorem{remark}{Remark}

\newtheorem{conjecture}{\bf Conjecture}
\newtheorem{theorem}{\bf Theorem}

\theoremstyle{definition}

\numberwithin{equation}{section}

\title[Critical point equation]{On Critical Point Equation of Compact Manifolds with Zero radial Weyl Curvature}

\author{H. Baltazar}

\address[H. Baltazar]{Departamento de Matem\'{a}tica, Universidade Federal do Piau\'{\i}\\
64049-550 Te\-re\-si\-na, Piau\'{\i}, Brazil.}
\email{halyson@ufpi.edu.br}

\date{September 05, 2017}

\begin{document}

\newcommand{\spacing}[1]{\renewcommand{\baselinestretch}{#1}\large\normalsize}
\spacing{1.2}

\begin{abstract}
Let $\mathcal{C}$ be the space of smooth metrics $g$ on a given compact ma\-ni\-fold $M^{n}$ ($n\geq3$) with constant scalar curvature and unitary volume.  The goal of this paper is to study the critical point of the total scalar curvature functional restricted to the space $\mathcal{C}$ (we shall refer to this critical point as CPE metrics) under assumption that $(M,g)$  has zero radial Weyl curvature. Among the results obtained, we emphasize that in 3-dimension we will be able to prove that a CPE metric with nonnegative sectional curvature must be isometric to a standard $3$-sphere. We will also prove that a $n$-dimensional, $4\leq n\leq10,$ CPE metric satisfying a $L^{n/2}$-pinching condition will be isometric to a standard sphere. In addition, we shall conclude that such critical metrics are isometrics to a standard sphere under fourth-order vanishing condition on the Weyl tensor.
\end{abstract}

\maketitle

\section{Introduction}\label{intro}
Throughout this paper, we always assume that $M$ is an $n$-dimensional compact (without boundary) oriented Riemannian manifold with dimension at least three.
With this consideration, let $\mathcal{M}$ the set of Riemannian metrics on $M^{n}$ of volume 1, and $\mathcal{C}\subset\mathcal{M}$ the subset of Riemannian metrics with constant scalar curvature. We define the total scalar curvature functional $\mathcal{R}:\mathcal{M}\rightarrow\mathbb{R}$, as follows
\begin{equation}\label{scal}
\mathcal{R}(g)=\int_{M}R_{g}dM_{g}.
\end{equation}
It is well-known that critical point of this functional are precisely Einstein metrics, see for instance \cite[Chapter 4]{besse}.

Now, if we consider the functional in $\eqref{scal}$ restricted to $\mathcal{C},$ it is not difficult to see that, the Euler-Lagrangian equation is given by
\begin{equation}\label{cpeint}
Ric-\frac{R}{n}g=Hessf-\left(Ric-\frac{R}{n-1}g\right)f,
\end{equation}
for some smooth function $f$ defined on $M^{n}.$ Here, $Ric,\;R$ and $Hess$ stand for the Ricci tensor, the scalar curvature and the Hessian form on $M^{n},$ respectively. Moreover, taking the trace in (\ref{cpeint}) we obtain
$$\Delta f+\frac{R}{n-1}f=0.$$
In particular, $f$ is an eigenfunction of the Laplacian and, since the Laplacian has non-positive spectrum, we may conclude that the scalar curvature $R$ must be po\-si\-ti\-ve.

We also notice that, if $f$ is constant, then $f=0$ and the metric must be Einstein. Hence, from now on, we consider only the case when $(M^{n},g,f)$ is a non-trivial solution of the Equation (\ref{cpeint}).

Following the terminology used in \cite{BR14,BR15,CHY10,CHY14,Hwang2000} we recall the definition of CPE metrics.

\begin{definition}
\label{def1} A CPE metric is a 3-tuple $(M^n,\,g,\,f),$ where $(M^{n},\,g),$ is a compact oriented Riemannian manifold of dimension at least three with constant scalar curvature and $f: M^{n}\to \Bbb{R}$ is a non-constant smooth function satisfying equation (\ref{cpeint}). Such a function $f$ is called a potential function.
\end{definition}

It was conjectured in 1980's that a critical metric of the total scalar curvature functional, restricted to the space $\mathcal{C}$ must be Einstein. Moreover, if a non-trivial solution $(M^{n},g,f)$ of (\ref{cpeint}) is Einstein then, after to apply the Obata's theorem (cf. \cite{Obata}), we can deduce that $(M^{n},g)$ is isometric to a standard sphere. The conjecture was proposed in \cite{besse} and here we will present this problem in the following way.

\begin{conjecture}\label{conjCPE}
A CPE metric is always Einstein.
\end{conjecture}

The proof of the Conjecture~\ref{conjCPE} has been a subject of interest for many authors. In \cite{lafontaine}, Lafontaine showed that the Conjecture~\ref{conjCPE} is true for a locally conformally flat manifold. Later on, this result was improved by Chang, Hwang and Yun under harmonic curvature assumption which is clearly weaker than locally conformally flat condition considered in Lafontaine's result (see \cite{CHY14,CHY16} for a complete solution). This same authors in \cite{CHY12}, were able to solve the conjecture for a manifold satisfying the parallel Ricci tensor condition. In \cite{Hwang2000}, Hwang proved the CPE conjecture provided $f\geq-1.$ At the same time, with a suitable constrain, Hwang was able to conclude that, if the second homology group of a $3$-manifold vanishes, then it is diffeomorphic to $\mathbb{S}^{3}$ (see \cite{Hwang03}). Meanwhile, Ribeiro Jr. and Barros in \cite{BR14} showed that the conjecture is also true for $4$-dimensional half conformally flat manifolds. It is important to say that this result was improved by the same authors and Leandro in \cite{BR15} under harmonicity of the self-dual part of Weyl tensor. In despite some important progresses, it remains a big challenge to prove the Besse's conjecture. For more references on CPE metrics, see \cite{Benjamim,CHY10,Hwang13,Leandro15,QY,Alex} and references therein.

Before presenting our first result, it is fundamental to remember that a Riemannian manifold $(M^{n},\,g)$ has {\it zero radial Weyl curvature} when, for a suitable potential function $f$ on $M^n,$ $W(\,\cdot\,,\,\cdot\,,\,\cdot\,,\nabla f)=0.$ This class of manifolds clearly includes the case of locally conformally flat manifolds. We would like to say that, this condition have been used to classify generalized quasi-Einstein manifolds (cf., for instance,  \cite{catino,PW,Bleandro} and \cite{pwylie}). Here, we shall use this condition to obtain the following result.

\begin{theorem}\label{nDCPE}
The CPE conjecture is true for $n$-dimensional ($n\geq3$) manifolds with nonnegative sectional curvature satisfying the zero radial Weyl curvature condition.
\end{theorem}

It is well known that, for $n=3,$ the Weyl tensor vanishes identically. Then, when we restrict to 3-dimensional case, it is easy to verify that a CPE metric with nonnegative sectional curvature must be isometric to a standard 3-sphere. In fact, we establish the following.

\begin{corollary}\label{3DCPE}
The CPE conjecture is true for $3$-dimensional manifolds with nonnegative sectional curvature.
\end{corollary}

Now, motivated by the recent work on critical metrics of the volume functional and positive static triples due to first author and Ribeiro Jr. (see \cite{balt17}, Corollaries 1 and 2), see also the works \cite{Amb,Batista17} for the three-dimensional case, we shall obtain similar result for CPE metrics satisfying the zero radial Weyl curvature assumption. More precisely, we have:

\begin{theorem}\label{CPEokumura}
Let $(M^{n},g,f),$ $n\geq3,$ be a CPE metric with zero radial Weyl curvature satisfying
$$|\mathring{Ric}|^{2}\leq\frac{R^{2}}{n(n-1)}.$$
Then $(M,g)$ is isometric to a standard sphere.
\end{theorem}

As an immediate consequence of Theorem~\ref{CPEokumura} we get the following result in the three-dimensional case.

\begin{corollary}
Let $(M^{3},g,f)$ be a CPE metric satisfying
$$|\mathring{Ric}|^{2}\leq\frac{R^{2}}{6}.$$
Then $(M^{3},g)$ is isometric to a standard sphere.
\end{corollary}

In order to proceed, let us introduce the definition of Yamabe constant associated to a Riemannian manifold $(M^{n},g)$. It is defined by
\begin{eqnarray*}
\mathcal{Y}(M,[g])&=&\inf_{\widetilde{g}\in[g]}\frac{\int_{M}\widetilde{R}dM_{\widetilde{g}}}{(\int_{M}dM_{\widetilde{g}})^{\frac{n-2}{n}}}\\
&=&\frac{4(n-1)}{n-2}\inf_{u\in W^{1,2}(M)}\frac{\int_{M}|\nabla u|^{2}dM_{g}+\frac{n-2}{4(n-1)}\int_{M}Ru^{2}dM_{g}}{(\int_{M}|u|^{\frac{2n}{n-2}}dM_{g})^{\frac{n-2}{n}}},
\end{eqnarray*}
where $[g]$ is the conformal class of the metric $g.$

The next result was inspired by the recent work of Catino \cite{Catino16}, where the author showed that an $n$-dimensional, $4\leq n\leq 6,$ compact gradient shrinking Ricci soliton satisfying a $L^{n/2}$-pinching condition must be isometric to a quotient of the round sphere. More recently, Huang in \cite{huang} generalized the result proved by Catino for a class of manifolds well known as $\rho$-Einstein solitons (see \cite[Theorem 1.1]{huang} for more details). We refer the readers to \cite{Catino13,FU,FuXiao17,FuXiao,FX,HV} for more results on this subject. Here, we shall obtain a similar pinching assumption in order to conclude that the Besse's conjecture must be true. More precisely, we have the following result.

\begin{theorem}\label{CPEyamabe}
Let $(M^{n},g,f),$ $4\leq n\leq10,$ be a CPE metric with zero radial Weyl curvature satisfying
\begin{equation}\label{intP}
\left(\int_{M}\left|W+\frac{\sqrt{n}}{\sqrt{2}(n-2)}Ric\varowedge g\right|^{\frac{n}{2}}dM_{g}\right)^{\frac{2}{n}}\leq\sqrt{\frac{n-2}{72(n-1)}}\mathcal{Y}(M,[g]).
\end{equation}
Then $(M^{n},g)$ is isometric to a standard sphere.
\end{theorem}

In this last part of the paper we shall focus attention on CPE metrics with vanish condition on Weyl tensor. So, in the same spirit of the recent work due by Catino, Mastrolia and Monticelli \cite{cmm}, let us introduce the following definitions
$$div^{4}W=\nabla_{k}\nabla_{j}\nabla_{i}\nabla_{l}W_{ijkl}$$
and
$$div^{3}C=\nabla_{k}\nabla_{j}\nabla_{i}C_{ijk},$$
where $W$ and $C$ are the Weyl and the Cotton tensors, respectively (for more details about this tensor, see Section \ref{Preliminaries}).

In \cite{QY}, Qing and Yuan  studied the $3$-dimensional CPE metric satisfying $div^{3}C=0.$ More precisely, the authors showed that a three dimensional CPE metric with third order divergence-free Cotton tensor (i.e., $div^{3}C=0$) must be isometric to the standard $3$-sphere. Recently, based in the ideas developed in \cite{cmm}, Santos obtained a positive answer for Conjecture~\ref{conjCPE} under the second order divergence-free Weyl tensor condition (cf. \cite[Theorem 1]{Alex} for more details). Thus, inspired in the previous results, it is natural to ask what happens in higher dimension for a CPE metric satisfying the fourth order divergence-free Weyl tensor condition. In order to do so, we shall provide an integral formula (see Section~\ref{div4W}, Proposition~\ref{div3CCaux}) which will allow us to prove the following result.

\begin{theorem}\label{CPEdiv4W}
Let $(M^{n},g,f),$ $n\geq4,$ be a CPE metric with zero radial Weyl curvature satisfying $div^{4}W=0.$
Then $(M^{n},g)$ is isometric to a standard sphere.
\end{theorem}

\section{Preliminaries}
\label{Preliminaries}

Throughout this section we recall some informations and basic results that will be useful in the proof of our main results. Firstly, we note that the fundamental equation of a CPE metric (\ref{cpeint}) can be rewrite as
\begin{equation}\label{eq1:CPE}
(1+f)Ric=Hessf+\frac{R}{n}g+\frac{Rf}{n-1}g.
\end{equation}
As we saw in the introduction, tracing (\ref{eq1:CPE}) we have
\begin{equation}\label{eqtrace}
\Delta f+\frac{R}{n-1}f=0.
\end{equation}
Furthermore, by using (\ref{eqtrace}) it is not difficult to check that
\begin{equation}
\label{IdRicHess} (1+f)\mathring{Ric}=\mathring{Hess f},
\end{equation}
where $\mathring{T}$ stands for the traceless of $T.$

For sake of simplicity, we now rewrite Equation (\ref{cpeint}) in the tensorial language as follows
\begin{equation}\label{eq:tensorial}
R_{ij}-\frac{R}{n}g_{ij}=\nabla_{i}\nabla_{j}f-\left(R_{ij}-\frac{R}{n-1}g_{ij}\right)f.
\end{equation}

Let us recall three special tensors in the study of curvature for a Riemannian manifold $(M^n,\,g),\,n\ge 3.$  The first one is the Weyl tensor $W$ which is defined by the following decomposition formula
\begin{eqnarray}\label{weyl}
R_{ijkl}&=&W_{ijkl}+(A\varowedge g)_{ijkl},
\end{eqnarray}
where $R_{ijkl}$ stands for the Riemann curvature operator $Rm$ and $A_{ij}$ is the Schouten tensor, defined by
\begin{equation}
\label{defschouten} A_{ij}=\frac{1}{n-2}\left(R_{ij}-\frac{R}{2(n-1)}g_{ij}\right).
\end{equation}
Moreover, the symbol $\varowedge$ denotes the Kulkarni-Nomizu product which is defined for any two symmetric $(0,2)$-tensors $S$ and $T$ as follows
\begin{equation*}
(S\varowedge T)_{ijkl}=S_{ik}T_{jl}+S_{jl}T_{ik}-S_{il}T_{jk}-S_{jk}T_{il}.
\end{equation*}

The second tensor is the Cotton tensor $C$ given by
\begin{equation*}
\displaystyle{C_{ijk}=(n-2)(\nabla_{i}A_{jk}-\nabla_{j}A_{ik})},
\end{equation*}
which clearly becomes
\begin{equation}\label{cotton}
\displaystyle{C_{ijk}=\nabla_{i}R_{jk}-\nabla_{j}R_{ik}-\frac{1}{2(n-1)}\big(\nabla_{i}Rg_{jk}-\nabla_{j}R g_{ik}).}
\end{equation}
Note that $C_{ijk}$ is skew-symmetric in the first two indices and trace-free in any two indices. These two above tensors are related as follows
\begin{equation}\label{cottonwyel}
\displaystyle{C_{ijk}=-\frac{(n-2)}{(n-3)}\nabla_{l}W_{ijkl},}
\end{equation}provided $n\ge 4.$ As a consequence, we have the following identity:
\begin{equation}\label{div3Cdiv4W}
\displaystyle{{\rm div}^{3}C+\frac{(n-2)}{(n-3)}{\rm div}^{4}W=0}.
\end{equation}

Finally, the third tensor is the Bach tensor  which was introduced by Bach \cite{bach} in the study of conformal relativity. It is defined in terms of the components of the Weyl tensor $W_{ikjl}$ as follows
\begin{equation}
\label{bach} B_{ij}=\frac{1}{n-3}\nabla_{k}\nabla_{l}W_{ikjl}+\frac{1}{n-2}R_{kl}W_{ikjl},
\end{equation}
for $n\geq4.$ Using (\ref{cottonwyel}) it is immediate to observe that the following identity is true
\begin{equation}\label{bacha}
(n-2)B_{ij}=\nabla_{k}C_{kij}+W_{ikjl}R_{kl}.
\end{equation}
We say that a Riemannian manifold $(M^n,g)$ is Bach-flat when $B_{ij}=0.$

For our purpose, we remember that as consequence of Bianchi identity, we have
\begin{equation}\label{Bianchi}
(div Rm)_{jkl}=\nabla_{i}R_{ijkl}=\nabla_kR_{lj}-\nabla_lR_{kj}.
\end{equation}

In addition, we also recall that, from commutation formulas (Ricci identities), for any Riemannian manifold $M^n$ we have
\begin{equation}\label{idRicci1}
\nabla_i\nabla_j R_{pq}-\nabla_j\nabla_i R_{pq}=R_{ijps}R_{sq}+R_{ijqs}R_{ps},
\end{equation} for more details see \cite{chow,Via}.

Under these notations we have the following lemma. We refer the reader to \cite[Lemma 2.1]{BR14} for its proof.
\begin{lemma}\label{L1}
Let $(M^{n},g,f)$ be a CPE metric. Then we have:
\begin{equation*}
(1+f)(\nabla_{i}R_{jk}-\nabla_{j}R_{ik})=R_{ijkl}\nabla_{l}f+\frac{R}{n-1}(\nabla_{i}fg_{jk}-\nabla_{j}fg_{ik})-(\nabla_{i}fR_{jk}-\nabla_{j}f R_{ik}).
\end{equation*}
\end{lemma}

Now, for a CPE metric, we recall the following 3-tensor defined, for instance, in \cite{BR14,BR15,Alex},
\begin{eqnarray}\label{TensorT}
T_{ijk}&=&\frac{n-1}{n-2}(R_{ik}\nabla_{j}f-R_{jk}\nabla_{i}f)-\frac{R}{n-2}(g_{ik}\nabla_{j}f-g_{jk}\nabla_{i}f)\nonumber\\
&&+\frac{1}{n-2}(g_{ik}R_{js}\nabla_{s}f-g_{jk}R_{is}\nabla_{s}f).
\end{eqnarray}
Note that, $T_{ijk}$ has the same symmetry properties as the Cotton tensor:
$$T_{ijk}=-T_{jik}\;\;\;and\;\;\;g^{ij}T_{ijk}=g^{ik}T_{ijk}=0.$$
Furthermore, we remember that the tensor $T_{ijk}$ is related to the Cotton tensor $C_{ijk}$ and the Weyl tensor $W_{ijkl}$ by
\begin{equation}\label{CTW}
(1+f)C_{ijk}=T_{ijk}+W_{ijkl}\nabla_{l}f.
\end{equation}

To close this section, let us recall a result of Hwang (see, for instance, \cite{Hwang2000,Hwang03} for more details), which plays an important role in our theorems.

\begin{proposition}\label{propLS}
Let $(M^{n},g,f)$ be a CPE metric with $f$ non-constant. Then the set $\{x\in M; f(x)=-1\}$ has measure zero.
\end{proposition}

\section{Bochner type formulas and applications}
In this section we will prove the Theorems \ref{nDCPE} and \ref{CPEokumura} announced in the introduction. To begin with, we shall present a couple of divergence formulas which will be crucial for the proof of main results of this section.

\begin{lemma}\label{divRIC1}
Let $(M^{n},g,f)$ be a CPE metric. Then we have:
\begin{eqnarray*}
div((1+f)\nabla|Ric|^{2})&=&2(1+f)|\nabla Ric|^{2}-(1+f)|C_{ijk}|^{2}+2(1+f)\nabla_{i}(C_{ijk}R_{jk})\\
&&+\frac{2}{n-1}(1+f)R|\mathring{Ric}|^{2}+\langle\nabla f,\nabla |Ric|^{2}\rangle\\
&&+2(1+f)\left(\frac{n}{n-2}tr(\mathring{Ric}^{3})-W_{ijkl}\mathring{R}_{ik}\mathring{R}_{jl}\right),
\end{eqnarray*}
where $\mathring{Ric}^{3}$ is the $2$-tensor defined by $(\mathring{Ric}^{3})_{ij}=\mathring{R}_{ik}\mathring{R}_{kl}\mathring{R}_{lj}.$
\end{lemma}

\begin{proof}
Firstly, we shall obtain an expression for the laplacian of the squared norm of the Ricci tensor. Indeed, since the scalar curvature is constant, by direct computation we get
\begin{eqnarray*}
\Delta|Ric|^{2}&=&\nabla_{p}\nabla_{p}R_{ij}^{2}\\
&=&2|\nabla Ric|^{2}+2R_{ij}\nabla_{p}\nabla_{p}R_{ij}\\
&=&2|\nabla Ric|^{2}+2R_{ij}\nabla_{p}(C_{pij}+\nabla_{i}R_{pj}).
\end{eqnarray*}
Hence, from (\ref{idRicci1}) and the skew-symmetric property of Cotton tensor, we may rewrite this last expression as follows
\begin{eqnarray}\label{deltaric}
\Delta|Ric|^{2}&=&2|\nabla Ric|^{2}+2\nabla_{p}(C_{pij}R_{ij})-|C_{ijk}|^{2}+2R_{ij}\nabla_{p}\nabla_{i}R_{pj}\nonumber\\
&=&2|\nabla Ric|^{2}+2\nabla_{p}(C_{pij}R_{ij})-|C_{ijk}|^{2}\nonumber\\
&&+2(R_{ij}R_{ik}R_{jk}-R_{ijkl}R_{ik}R_{jl}),
\end{eqnarray}
where we change some indices for simplicity.

Next, after a straightforward computation using (\ref{weyl}) and the fact that $\mathring{R}_{ij}=R_{ij}-\frac{R}{n}g_{ij},$ we arrived at
\begin{equation}\label{auxidRM}
R_{ij}R_{ik}R_{jk}-R_{ijkl}R_{ik}R_{jl}=\frac{R}{n-1}|\mathring{Ric}|^{2}+\frac{n}{n-2}tr(\mathring{Ric}^{3})-W_{ijkl}\mathring{R}_{ik}\mathring{R}_{jl}
\end{equation}
(such identity can be found in \cite[Lemma 4]{balt17}), which replacing in (\ref{deltaric}) gives
\begin{eqnarray*}
\Delta|Ric|^{2}&=&2|\nabla Ric|^{2}-|C_{ijk}|^{2}+2\nabla_{p}(C_{pij}R_{ij})+\frac{2}{n-1}R|\mathring{Ric}|^{2}\\
&&+\frac{2n}{n-2}tr(\mathring{Ric}^{3})-2W_{ijkl}\mathring{R}_{ik}\mathring{R}_{jl}.
\end{eqnarray*}
To finish, it is suffices to observe that
\begin{eqnarray*}
div((1+f)\nabla |Ric|^{2})&=&(1+f)\Delta|Ric|^{2}+\langle\nabla f,\nabla|Ric|^{2}\rangle\\
&=&2(1+f)|\nabla Ric|^{2}-(1+f)|C_{ijk}|^{2}+2(1+f)\nabla_{p}(C_{pij}R_{ij})\\
&&+\frac{2}{n-1}(1+f)R|\mathring{Ric}|^{2}+\langle\nabla f,\nabla|Ric|^{2}\rangle\\
&&+2(1+f)\left(\frac{n}{n-2}tr(\mathring{Ric}^{3})-W_{ijkl}\mathring{R}_{ik}\mathring{R}_{jl}\right).
\end{eqnarray*}
So, the proof is completed.
\end{proof}

\begin{remark}
Notice that the above lemma must be true for an arbitrary $n$-dimensional Riemannian manifold with constant scalar curvature and a smooth function $f.$
\end{remark}

The next lemma provide another expression for $div((1+f)\nabla|Ric|^{2}).$

\begin{lemma}\label{lemmapsi}
Let $(M^{n},g,f)$ be a CPE metric. Then we have:
\begin{eqnarray*}
\frac{1}{2}div((1+f)\nabla|Ric|^{2})&=&-(1+f)|C_{ijk}|^{2}+(1+f)|\nabla Ric|^{2}+\langle\nabla f,\nabla |Ric|^{2}\rangle\\
&&+\frac{R}{n-1}|\mathring{Ric}|^{2}+2\nabla_{i}((1+f)C_{ijk}R_{jk}).
\end{eqnarray*}
\end{lemma}

\begin{proof}
In what follows, we denote by $\psi$ the function
$$\psi=\nabla_{i}(\nabla_{j}fR_{ik}R_{kj}+R_{ijkl}\nabla_{l}fR_{jk}).$$
Now, by Lemma~\ref{L1} and since we already know that $M$ has constant scalar curvature, we have
\begin{eqnarray}\label{eqnablai}
\psi&=&\nabla_{i}\Big((1+f)C_{ijk}R_{jk}+|Ric|^{2}\nabla_{i}f+\frac{R}{n-1}R_{ij}\nabla_{j}f-\frac{R^{2}}{n-1}\nabla_{i}f\Big).
\end{eqnarray}
In the sequel, by using the twice contracted second Bianchi identity the Eq. (\ref{eqnablai}) becomes
\begin{eqnarray*}
\psi&=&\nabla_{i}((1+f)C_{ijk}R_{jk})+\nabla_{i}(|Ric|^{2}\nabla_{i}f)+\frac{R}{n-1}R_{ij}\nabla_{i}\nabla_{j}f-\frac{R^{2}}{n-1}\Delta f.
\end{eqnarray*}
Substituting (\ref{eqtrace}) and (\ref{eq:tensorial}) in the above expression we get
\begin{eqnarray}\label{eqa}
\psi&=&\nabla_{i}((1+f)C_{ijk}R_{jk})+\nabla_{i}(|Ric|^{2}\nabla_{i}f)+\frac{R}{n-1}(1+f)|\mathring{Ric}|^{2}\nonumber \\
&&-\frac{R^{2}}{n}\Delta f,
\end{eqnarray}
i.e.,
\begin{eqnarray}\label{eqkey1}
\psi &=&\nabla_{i}((1+f)C_{ijk}R_{jk})+\langle\nabla|Ric|^{2},\nabla f\rangle+\frac{R}{n-1}|\mathring{Ric}|^{2}.
\end{eqnarray}

On the other hand,
\begin{eqnarray*}
\psi&=&\nabla_{i}\nabla_{j}fR_{ik}R_{kj}+\nabla_{j}fR_{ik}\nabla_{i}R_{kj}+\nabla_{i}R_{ijkl}R_{jk}\nabla_{l}f\\
&&+R_{ijkl}\nabla_{i}R_{jk}\nabla_{l}f+R_{ijkl}R_{jk}\nabla_{i}\nabla_{l}f,
\end{eqnarray*}
which from (\ref{cotton}), (\ref{Bianchi}) and jointly with symmetries properties of the Riemann tensor we may deduce
\begin{eqnarray*}
\psi&=&\nabla_{j}fR_{ik}(C_{ijk}+\nabla_{j}R_{ik})+C_{klj}R_{jk}\nabla_{l}f+\frac{1}{2}R_{ijkl}(\nabla_{i}R_{jk}-\nabla_{j}R_{ik})\nabla_{l}f\\
&&+\nabla_{i}\nabla_{j}fR_{ik}R_{kj}+R_{ijkl}R_{jk}\nabla_{i}\nabla_{l}f\\
&=&(\nabla_{j}fR_{ik}-\nabla_{i}fR_{jk})C_{ijk}+\frac{1}{2}\langle\nabla f,\nabla|Ric|^{2}\rangle+\frac{1}{2}R_{ijkl}C_{ijk}\nabla_{l}f\\
&&+\nabla_{i}\nabla_{j}fR_{ik}R_{kj}+R_{ijkl}R_{jk}\nabla_{i}\nabla_{l}f.
\end{eqnarray*}
Using Lemma \ref{L1} again and Eq. (\ref{eq:tensorial}) we arrived at
\begin{eqnarray}\label{eqkey2}
\psi&=&\frac{1}{2}(\nabla_{j}fR_{ik}-\nabla_{i}fR_{jk})C_{ijk}+\frac{1}{2}\langle\nabla f,\nabla|Ric|^{2}\rangle+\frac{1}{2}(1+f)|C_{ijk}|^{2}\nonumber\\
&&+(1+f)(R_{ij}R_{ik}R_{jk}-R_{ijkl}R_{ik}R_{jl}).
\end{eqnarray}
To proceed, note that from (\ref{cotton}) we can infer
\begin{eqnarray}\label{normC}
(1+f)|C_{ijk}|^{2}&=&2(1+f)(|\nabla Ric|^{2}-\nabla_{i}R_{jk}\nabla_{j}R_{ik})\nonumber\\
&=&2(1+f)|\nabla Ric|^{2}-2\nabla_{j}((1+f)\nabla_{i}R_{jk}R_{ik})+2\nabla_{i}R_{jk}\nabla_{j}fR_{ik}\nonumber\\
&&+2(1+f)\nabla_{j}\nabla_{i}R_{jk}R_{ik}.
\end{eqnarray}
Thus, inserting (\ref{cotton}) and (\ref{idRicci1}) into (\ref{normC}) yields
\begin{eqnarray*}
\frac{1}{2}(1+f)|C_{ijk}|^{2}&-&(1+f)|\nabla Ric|^{2}+\nabla_{j}((1+f)\nabla_{i}R_{jk}R_{ik})=\\
&=&\frac{1}{2}C_{ijk}(\nabla_{j}fR_{ik}-\nabla_{i}fR_{jk})+\frac{1}{2}\langle\nabla f,\nabla|Ric|^{2}\rangle\\
&&+(1+f)(R_{ij}R_{ik}R_{jk}-R_{ijkl}R_{ik}R_{jl}).
\end{eqnarray*}
So, combining this expression with (\ref{eqkey2}), it is immediate to check that
\begin{eqnarray}\label{eqkey3}
\psi&=&(1+f)|C_{ijk}|^{2}-(1+f)|\nabla Ric|^{2}+\nabla_{j}((1+f)\nabla_{i}R_{jk}R_{ik})\nonumber\\
&=&(1+f)|C_{ijk}|^{2}-(1+f)|\nabla Ric|^{2}+\nabla_{j}((1+f)C_{ijk}R_{ik})\nonumber\\
&&+\frac{1}{2}div((1+f)\nabla|Ric|^{2}).
\end{eqnarray}
Therefore, the desired divergent formula follows from (\ref{eqkey1}) and (\ref{eqkey3}).
\end{proof}

\begin{remark}\label{ricparallel}
Notice that, if we consider a CPE metric with parallel Ricci curvature, follows by Kato's inequality,
$$|\nabla|Ric||\leq|\nabla Ric|,$$
that $|Ric|$ is constant on $M^{n}.$ Consequently, after to take the integral in Eq. (\ref{eqa}), we may deduce that our manifold is Einstein and therefore it must be isometric to a round sphere. This result is already known and has been proven with a different technique by Chang, Hwang and Yun in \cite[Theorem 1.1]{CHY12}.
\end{remark}

Before we present our next lemma let us highlight that, integrating (\ref{eqa}) over $M$ and applying the divergence formula,  we may obtain a important integral identity which will be essential in order to get the next results, more precisely we have that
\begin{equation}\label{keyequality}
\int_{M}(1+f)R|\mathring{Ric}|^{2}dM_{g}=0.
\end{equation}

Now, as a consequence of this Bochner type formulas we can deduce the following integral formulas for CPE metrics.

\begin{lemma}\label{intnormRIC}
Let $(M^{n},g,f)$ be a CPE metric. Then we have:
\begin{eqnarray*}
\int_{M}|\mathring{Ric}|^{2}|\nabla f|^{2}dM_{g}&=&\int_{M}(1+f)^{2}|\nabla Ric|^{2}dM_{g}+\frac{n-3}{2(n-1)}\int_{M}(1+f)^{2}|C_{ijk}|^{2}dM_{g}\\
&&+\frac{2}{n-1}\int_{M}(1+f)^{2}R|\mathring{Ric}|^{2}dM_{g}-\frac{n-2}{n-1}\int_{M}(1+f)C_{ijk}W_{ijkl}\nabla_{l}fdM_{g}\\
&&+\int_{M}(1+f)^{2}\left(\frac{n}{n-2}tr(\mathring{Ric}^{3})-W_{ijkl}\mathring{R}_{ik}\mathring{R}_{jl}\right)dM_{g}.
\end{eqnarray*}
\end{lemma}

\begin{proof}
First of all, we multiply the equation obtained in Lemma~\ref{divRIC1} by $(1+f).$ Then, integrating by parts over $M$ we get
\begin{eqnarray}\label{auxLint}
-\int_{M}\langle \nabla(1+f)^{2}, \nabla|Ric|\rangle dM_{g}&=&2\int_{M}(1+f)^{2}|\nabla Ric|^{2}dM_{g}-\int_{M}(1+f)^{2}|C_{ijk}|^{2}dM_{g}\nonumber\\
&&-4\int_{M}(1+f)C_{ijk}\nabla_{i}fR_{jk}dM_{g}+\frac{2}{n-1}\int_{M}R(1+f)^{2}|\mathring{Ric}|^{2}dM_{g}\nonumber\\
&&+2\int_{M}(1+f)^{2}\left(\frac{n}{n-2}tr(\mathring{Ric}^{3})-W_{ijkl}\mathring{R}_{ik}\mathring{R}_{jl}\right)dM_{g}.
\end{eqnarray}
Next, since the Cotton tensor satisfies
\begin{eqnarray}\label{CID}
C_{ijk}\nabla_{i}fR_{jk}&=&\frac{1}{2}C_{ijk}(\nabla_{i}fR_{jk}-\nabla_{j}fR_{ik})\nonumber\\
&=&-\frac{n-2}{2(n-1)}C_{ijk}T_{ijk},
\end{eqnarray}
and taking into account Eq. (\ref{CTW}) together with fact that $M$ has constant scalar curvature, it is not difficult to verify that (\ref{auxLint}) becomes
\begin{eqnarray*}
\int_{M}|\mathring{Ric}|^{2}\Delta(1+f)^{2}dM_{g}&=&2\int_{M}(1+f)^{2}|\nabla Ric|^{2}dM_{g}-\int_{M}(1+f)^{2}|C_{ijk}|^{2}dM_{g}\\
&&+\frac{2n-4}{n-1}\int_{M}(1+f)C_{ijk}((1+f)C_{ijk}-W_{ijkl}\nabla_{l}f)dM_{g}\\
&&+\frac{2}{n-1}\int_{M}(1+f)^{2}R|\mathring{Ric}|^{2}dM_{g}\\
&&+2\int_{M}(1+f)^{2}\left(\frac{n}{n-2}tr(\mathring{Ric}^{3})-W_{ijkl}\mathring{R}_{ik}\mathring{R}_{jl}\right)dM_{g},
\end{eqnarray*}
i.e.,

\begin{eqnarray}\label{auxnormRIC}
\int_{M}|\mathring{Ric}|^{2}\Delta(1+f)^{2}dM_{g}&=&2\int_{M}(1+f)^{2}|\nabla Ric|^{2}dM_{g}+\frac{n-3}{n-1}\int_{M}(1+f)^{2}|C_{ijk}|^{2}dM_{g}\nonumber\\
&&-\frac{2n-4}{n-1}\int_{M}(1+f)C_{ijk}W_{ijkl}\nabla_{l}fdM_{g}\nonumber\\
&&+2\int_{M}(1+f)^{2}\left(\frac{n}{n-2}tr(\mathring{Ric}^{3})-W_{ijkl}\mathring{R}_{ik}\mathring{R}_{jl}\right)dM_{g}\nonumber\\
&&+\frac{2}{n-1}\int_{M}(1+f)^{2}R|\mathring{Ric}|^{2}dM_{g}.
\end{eqnarray}

Now, notice that
\begin{equation}\label{delaf}
\Delta(1+f)^{2}=-2(1+f)\frac{Rf}{n-1}+2|\nabla f|^{2}.
\end{equation}
This data substituted in (\ref{auxnormRIC}) jointly with (\ref{keyequality}) gives the requested result.
\end{proof}

In the following, we will proceed in a similar way to the previous lemma, now using Lemma~\ref{lemmapsi}, to deduce another formula for $\int_{M}|\mathring{Ric}|^{2}|\nabla f|^{2}dM_{g}.$

\begin{lemma}\label{intnormRIC1}
Let $(M^{n},g,f)$ be a CPE metric. Then we have:
\begin{eqnarray*}
\int_{M}|\mathring{Ric}|^{2}|\nabla f|^{2}dM_{g}&=&\frac{1}{n-1}\int_{M}(1+f)^{2}R|\mathring{Ric}|^{2}dM_{g}-\frac{2}{3(n-1)}\int_{M}(1+f)^{2}|C_{ijk}|^{2}dM_{g}\\
&&+\frac{2}{3}\int_{M}(1+f)^{2}|\nabla Ric|^{2}dM_{g}-\frac{2(n-2)}{3(n-1)}\int_{M}(1+f)C_{ijk}W_{ijkl}\nabla_{l} fdM_{g}.
\end{eqnarray*}
\end{lemma}

\begin{proof}
To begin with, we use the Lemma~\ref{lemmapsi} and Eq. (\ref{keyequality}) to infer
\begin{eqnarray*}
-\frac{3}{4}\int_{M}\langle \nabla(1+f)^{2}, \nabla|Ric|\rangle dM_{g}&=&\int_{M}(1+f)^{2}|\nabla Ric|^{2}dM_{g}-\int_{M}(1+f)^{2}|C_{ijk}|^{2}dM_{g}\\
&&+2\int_{M}(1+f)\nabla_{i}((1+f)C_{ijk}R_{jk})dM_{g}\\
&=&\int_{M}(1+f)^{2}|\nabla Ric|^{2}dM_{g}-\int_{M}(1+f)^{2}|C_{ijk}|^{2}dM_{g}\\
&&-2\int_{M}(1+f)C_{ijk}\nabla_{i}fR_{jk}dM_{g}.
\end{eqnarray*}
Thus, with a straightforward computation using (\ref{CID}) and  (\ref{CTW}), we may apply integration by parts to deduce
\begin{eqnarray*}
\frac{3}{4}\int_{M}|\mathring{Ric}|^{2}\Delta(1+f)^{2}dM_{g}&=&-\frac{1}{n-1}\int_{M}(1+f)^{2}|C_{ijk}|^{2}dM_{g}+\int_{M}(1+f)^{2}|\nabla Ric|^{2}dM_{g}\\
&&-\frac{n-2}{n-1}\int_{M}(1+f)C_{ijk}W_{ijkl}\nabla_{l}fdM_{g}.
\end{eqnarray*}
To conclude the proof, it is sufficient to substitute (\ref{delaf}) in the above expression and then to use (\ref{keyequality}).
\end{proof}

With these considerations in mind, we are in position to prove our Theorems \ref{nDCPE} and \ref{CPEokumura}.

\subsection{Proof of Theorem~\ref{nDCPE}}
\begin{proof}
We compare the expressions obtained in Lemma~\ref{intnormRIC} and Lemma~\ref{intnormRIC1} to deduce
\begin{eqnarray}\label{KeyID}
0&=&\int_{M}(1+f)^{2}|\nabla Ric|^{2}dM_{g}+\frac{3n-5}{2(n-1)}\int_{M}(1+f)^{2}|C|^{2}dM_{g}\nonumber\\
&&+3\int_{M}(1+f)^{2}\left(\frac{n}{n-2}tr(\mathring{Ric}^{3})-W_{ijkl}\mathring{R}_{ik}\mathring{R}_{jl}\right)dM_{g}\nonumber\\
&&-\frac{n-2}{n-1}\int_{M}(1+f)C_{ijk}W_{ijkl}\nabla_{l}fdM_{g}\nonumber\\
&&+\frac{3}{n-1}\int_{M} (1+f)^{2}R|\mathring{Ric}|^{2}dM_{g}.
\end{eqnarray}
Hence, taking into account our assumption that $M^{n}$ has zero radial Weyl curvature, it is easy to verify that (\ref{KeyID}) becomes
\begin{eqnarray}\label{KeyID2}
0&=&\int_{M}(1+f)^{2}|\nabla Ric|^{2}dM_{g}+\frac{3n-5}{2(n-1)}\int_{M}(1+f)^{2}|C|^{2}dM_{g}\nonumber\\
&&+3\int_{M}(1+f)^{2}\left(\frac{n}{n-2}tr(\mathring{Ric}^{3})-W_{ijkl}\mathring{R}_{ik}\mathring{R}_{jl}+\frac{1}{n-1}R|\mathring{Ric}|^{2}\right)dM_{g}.
\end{eqnarray}

In order to proceed, we need the following pointwise estimate which is satisfied in general for every metric with nonnegative sectional curvature, namely,
\begin{equation}\label{INEricci}
R_{ij}R_{ik}R_{jk}\geq R_{ijkl}R_{ik}R_{jl}.
\end{equation}
In fact, let $\{e_{i}\}_{i=1}^{n}$ be the eigenvectors of $Ric$  and let $\lambda_{i}$ be the corresponding eigenvalues. Then, by direct computation, we have that
\begin{eqnarray*}
R_{ij}R_{ik}R_{jk}-R_{ijkl}R_{ik}R_{jl}&=&\sum_{i}\lambda_{i}^{3}-\sum_{i,j}R_{ijij}\lambda_{i}\lambda_{j}\\
&=&\sum_{i,j}\lambda_{i}^{2}R_{ijij}-\sum_{i,j}R_{ijij}\lambda_{i}\lambda_{j}\\
&=&\frac{1}{2}\sum_{i,j}(\lambda_{i}-\lambda_{j})^{2}R_{ijij},
\end{eqnarray*}
i.e.,
\begin{eqnarray*}
R_{ij}R_{ik}R_{jk}=R_{ijkl}R_{ik}R_{jl}+\frac{1}{2}\sum_{i,j}(\lambda_{i}-\lambda_{j})^{2}K_{ij},
\end{eqnarray*}
where $K_{ij}$ is the sectional curvature defined by the two-plane spanned by $e_{i}$ and $e_{j}.$ Therefore, as we are supposing the metric with nonnegative sectional curvature, clearly we get the desired inequality.

So, from (\ref{auxidRM}) and (\ref{INEricci}), we may deduce that each term in (\ref{KeyID2}) must be nonnegative. In  particular, we conclude, jointly with Proposition~\ref{propLS}, that $M^{n}$ has parallel Ricci curvature. Now the result follows from Remark~\ref{ricparallel} (see also \cite[Theorem 1.1]{CHY12}).
\end{proof}

\subsection{Proof of Theorem~\ref{CPEokumura}}

\begin{proof}
Firstly, we recall the classical Okumura's inequality which, in few words, it says that the following inequality is true
$$tr(\mathring{Ric}^{3})\geq -\frac{n-2}{\sqrt{n(n-1)}}|\mathring{Ric}|^{3},$$
for more details see \cite[Lemma 2.1]{Okumura}. Consequently, we get
\begin{equation}\label{OKaux}
\frac{n}{n-2}tr(\mathring{Ric}^{3})+\frac{1}{n-1}R|\mathring{Ric}|^{2}\geq\frac{n}{\sqrt{n(n-1)}}\left(\frac{R}{\sqrt{n(n-1)}}-|\mathring{Ric}|\right)|\mathring{Ric}|^{2}.
\end{equation}

Otherwise, considering the $n$-dimensional case with $n\geq4,$ we first deduce a key integral identity to achieve our goals. More precisely, using (\ref{eq1:CPE}) and after a simple integration by parts, we may deduce
\begin{eqnarray*}
\int_{M}(1+f)^{2}W_{ijkl}R_{ik}R_{jl}dM_{g}&=&\int_{M}(1+f)W_{ijkl}\nabla_{i}\nabla_{k}fR_{jl}dM_{g}\\
&=&-\int_{M}\nabla_{i}fW_{ijkl}\nabla_{k}fR_{jl}dM_{g}\\
&&-\int_{M}(1+f)\nabla_{i}W_{ijkl}\nabla_{k}fR_{jl}dM_{g}\\
&&-\int_{M}(1+f)W_{ijkl}\nabla_{k}f\nabla_{i}R_{jl}dM_{g}.
\end{eqnarray*}
Consequently, from (\ref{cottonwyel}) and the fact that $M^{n}$ has zero radial Weyl curvature, we obtain
\begin{eqnarray*}
\int_{M}(1+f)^{2}W_{ijkl}R_{ik}R_{jl}dM_{g}&=&\frac{n-3}{n-2}\int_{M}(1+f)C_{ijk}\nabla_{j}fR_{ik}dM_{g},
\end{eqnarray*}
where we change some indices for simplicity. Hence, as the Cotton tensor is skew-symmetric in the first two indices, see (\ref{CID}), we immediately have
\begin{eqnarray}\label{keyINTW}
\int_{M}(1+f)^{2}W_{ijkl}R_{ik}R_{jl}dM_{g}&=&\frac{n-3}{2(n-1)}\int_{M}(1+f)C_{ijk}T_{ijk}dM_{g}.
\end{eqnarray}
Note that, as $W_{ijkl}=0$ when $n=3,$ the above identity is clearly true for the three-dimensional case.
Now, we substitute (\ref{keyINTW}) in (\ref{KeyID}) to arrive at
\begin{eqnarray}\label{intweylradial}
0&=&\int_{M}(1+f)^{2}|\nabla Ric|^{2}dM_{g}+\frac{2}{n-1}\int_{M}(1+f)^{2}|C_{ijk}|^{2}dM_{g}\nonumber\\
&&+\frac{3n}{n-2}\int_{M}(1+f)^{2}tr(\mathring{Ric}^{3})dM_{g}+\frac{3}{n-1}\int_{M} (1+f)^{2}R|\mathring{Ric}|^{2}dM_{g},
\end{eqnarray}
which combined with (\ref{OKaux}) and our pinching assumption gives
\begin{eqnarray*}
0&\geq&\int_{M}(1+f)^{2}|\nabla Ric|^{2}dM_{g}+\frac{2}{n-1}\int_{M}(1+f)^{2}|C_{ijk}|^{2}dM_{g}\\
&&+\frac{3n}{\sqrt{n(n-1)}}\int_{M}(1+f)^{2}\left(\frac{R}{\sqrt{n(n-1)}}-|\mathring{Ric}|\right)|\mathring{Ric}|^{2}dM_{g}\geq0
\end{eqnarray*}
(we observe that the Equation (\ref{intweylradial}) must be true for all CPE metric satisfying zero radial Weyl condiction).
Therefore, by the first integral in the above expression jointly with proposition~\ref{propLS}, we obtain that $M^{n}$ has parallel Ricci curvature. Again the result follows from Remark~\ref{ricparallel}.
\end{proof}

\section{Integral Pinching Condition for CPE metric}

In this section we will investigate CPE metrics satisfying a $L^{n/2}$-pinching condition. In this sense, we will be able to prove Theorem~\ref{CPEyamabe} mentioned in Section~\ref{intro}. To do so, we shall present the following estimate for an arbitrary $n$-dimensional Riemannian manifold.  Such estimate appear, for instance, in \cite[Lemma 2.5]{FU} and here, we present its prove for convenience of the readers.

\begin{lemma}\label{Keyestimte}
On every $n$-dimensional Riemannian manifold the following estimate holds
\begin{eqnarray*}
\left|\frac{n}{n-2}\mathring{R}_{ij}\mathring{R}_{jk}\mathring{R}_{ik}-W_{ijkl}\mathring{R}_{ik}\mathring{R}_{jl}\right|\leq\sqrt{\frac{n-2}{2(n-1)}}\left(|W|^{2}+\frac{2n}{n-2}|\mathring{Ric}|^{2}\right)^{1/2}|\mathring{Ric}|^{2}.
\end{eqnarray*}
\end{lemma}
\begin{proof}
We follow \cite[Proposition 2.1]{Catino16}. Firstly, after some computation, we can deduce
\begin{eqnarray}\label{auxid}
\frac{n}{n-2}\mathring{R}_{ij}\mathring{R}_{jk}\mathring{R}_{ik}-W_{ijkl}\mathring{R}_{ik}\mathring{R}_{jl}&=&-\frac{1}{4}W_{ijkl}(\mathring{Ric}\varowedge\mathring{Ric})_{ijkl}\nonumber\\
&&-\frac{n}{8(n-2)}(\mathring{Ric}\varowedge g)_{ijkl}(\mathring{Ric}\varowedge\mathring{Ric})_{ijkl}\nonumber\\
&=&-\frac{1}{4}\left(W+\frac{n}{2(n-2)}\mathring{Ric}\varowedge g\right)_{ijkl}(\mathring{Ric}\varowedge\mathring{Ric})_{ijkl}.
\end{eqnarray}
Before proceeding, remember that $\mathring{Ric}\varowedge\mathring{Ric}$ has the same symmetries of the Riemann tensor, so it can be orthogonally decomposed as
$$\mathring{Ric}\varowedge\mathring{Ric}=T+V+U,$$
where $T$ is totally trace-free,
$$V_{ijkl}=-\frac{2}{n-2}(\mathring{Ric^{2}}\varowedge g)_{ijkl}+\frac{2}{n(n-2)}|\mathring{Ric}|^{2}(g\varowedge g)_{ijkl}$$
and
$$U_{ijkl}=-\frac{1}{n(n-1)}|\mathring{Ric}|^{2}(g\varowedge g)_{ijkl},$$
where $(\mathring{Ric^{2}})_{ij}=\mathring{R}_{ip}\mathring{R}_{pj}.$ Furthermore, it is not difficult to check the following identities
$$|\mathring{Ric}\varowedge\mathring{Ric}|^{2}=8(|\mathring{Ric}|^{4}-|\mathring{Ric^{2}}|^{2}),$$
$$|V|^{2}=\frac{16}{n-2}(|\mathring{Ric^{2}}|^{2}-\frac{1}{n}|\mathring{Ric}|^{4})$$
and
$$|U|^{2}=\frac{8}{n(n-1)}|\mathring{Ric}|^{4}.$$

In the sequel, we deduce
\begin{eqnarray*}
\left|\left(W+\frac{n}{2(n-2)}\mathring{Ric}\varowedge g\right)(\mathring{Ric}\varowedge\mathring{Ric})\right|^{2}&=&\left|\left(W+\frac{n}{2(n-2)}\mathring{Ric}\varowedge g\right)(T+V)\right|^{2}\\
&=&\left|\left(W+\sqrt{\frac{2}{n}}\frac{n}{2(n-2)}\mathring{Ric}\varowedge g\right)\left(T+\sqrt{\frac{n}{2}}V\right)\right|^{2}\\
\end{eqnarray*}
and it follows that
\begin{eqnarray*}
\left|\left(W+\frac{n}{2(n-2)}\mathring{Ric}\varowedge g\right)(\mathring{Ric}\varowedge\mathring{Ric})\right|^{2}&\leq&\left|W+\frac{\sqrt{n}}{\sqrt{2}(n-2)}\mathring{Ric}\varowedge g\right|^{2}(|T|^{2}+\frac{n}{2}|V|^{2})\\
&=&\frac{8(n-2)}{n-1}\left(|W|^{2}+\frac{n}{2(n-2)^{2}}|\mathring{Ric}\varowedge g|^{2}\right)|\mathring{Ric}|^{4},
\end{eqnarray*}
where we used that
$$|T|^{2}+\frac{n}{2}|V|^{2}=|\mathring{Ric}\varowedge\mathring{Ric}|^{2}+\frac{n-2}{2}|V|^{2}-|U|^{2}=\frac{8(n-2)}{n-1}|\mathring{Ric}|^{4}.$$

Now, since $|\mathring{Ric}\varowedge g|^{2}=4(n-2)|\mathring{Ric}|^{2},$ we obtain
\begin{eqnarray*}
\left|\left(W+\frac{n}{2(n-2)}\mathring{Ric}\varowedge g\right)(\mathring{Ric}\varowedge\mathring{Ric})\right|^{2}\leq \frac{8(n-2)}{n-1}\left(|W|^{2}+\frac{2n}{(n-2)}|\mathring{Ric}|^{2}\right)|\mathring{Ric}|^{4}.
\end{eqnarray*}

Finally, returning to Eq. (\ref{auxid}) jointly with above inequality to conclude our estimate.
\end{proof}

Proceeding, as explained in Section~\ref{intro}, every non-trivial solution to CPE metric has positive scalar curvature, so its Yamabe constant must be positive. To be precise, it is well known that, if $M$ is compact, the Yamabe constant  $\mathcal{Y}(M,[g])$ is positive if and only if there exists a conformal metric in $[g]$ with everywhere positive scalar curvature. Therefore, since $\mathcal{Y}(M,[g])>0,$ it is immediate to verify the following Yamabe-Sobolev inequality
\begin{eqnarray}\label{yamsol}
\mathcal{Y}(M,[g])\left(\int_{M}|u|^{\frac{2n}{n-2}}dM_{g}\right)^{\frac{n-2}{n}}&\leq&\frac{4(n-1)}{n-2}\int_{M}|\nabla u|^{2}dM_{g}\nonumber\\
&&+\int_{M}Ru^{2}dM_{g},
\end{eqnarray}
for every $u\in W^{1,2}(M).$

Now, after this preliminaries remarks, we are in position to prove our $L^{n/2}$-pinching result.

\subsection{Proof of Theorem \ref{CPEyamabe}}

\begin{proof}
In what follows, we take $u=(1+f)|\mathring{Ric}|$ in (\ref{yamsol}). Denoting
\begin{equation}\label{auxA}
\phi=\left(\int_{M}|(1+f)|\mathring{Ric}||^{\frac{2n}{n-2}}dM_{g}\right)^{\frac{n-2}{n}},
\end{equation}
it is immediate to verify that
\begin{eqnarray*}
\frac{n-2}{4(n-1)}\phi\mathcal{Y}(M,[g])&\leq&\int_{M}|\nabla((1+f)|\mathring{Ric}|)|^{2}dM_{g}+\frac{n-2}{4(n-1)}\int_{M}R(1+f)^{2}|\mathring{Ric}|^{2}dM_{g}\\
&=&\int_{M}||\mathring{Ric}|\nabla f+(1+f)\nabla|\mathring{Ric}||^{2}dM_{g}\\
&&+\frac{n-2}{4(n-1)}\int_{M}R(1+f)^{2}|\mathring{Ric}|^{2}dM_{g}\\
&=&\int_{M}|\nabla f|^{2}|\mathring{Ric}|^{2}dM_{g}+\frac{1}{2}\int_{M}\langle\nabla(1+f)^{2},\nabla|\mathring{Ric}|^{2}\rangle dM_{g}\\
&&\int_{M}(1+f)^{2}|\nabla|\mathring{Ric}||^{2}dM_{g}+\frac{n-2}{4(n-1)}\int_{M}R(1+f)^{2}|\mathring{Ric}|^{2}dM_{g}.
\end{eqnarray*}
Then, using integration by parts together with (\ref{delaf}) and (\ref{keyequality}) we can rewrite the above expression as
\begin{eqnarray*}
\frac{n-2}{4(n-1)}\phi\mathcal{Y}(M,[g])&\leq&\frac{n+2}{4(n-1)}\int_{M}(1+f)^{2}R|\mathring{Ric}|^{2}dM_{g}+\int_{M}(1+f)^{2}|\nabla|\mathring{Ric}||^{2}dM_{g}.
\end{eqnarray*}
Now, from Kato's inequality and the fact that $M$ has zero radial Weyl curvature, we can use (\ref{KeyID}) to deduce
\begin{eqnarray*}
\frac{n-2}{4(n-1)}\phi\mathcal{Y}(M,[g])&\leq&\frac{n-10}{4(n-1)}\int_{M}(1+f)^{2}R|\mathring{Ric}|^{2}dM_{g}\\
&&-\frac{3n-5}{2(n-1)}\int_{M}(1+f)^{2}|C_{ijk}|^{2}dM_{g}\\
&&-3\int_{M}(1+f)^{2}\left(\frac{n}{n-2}tr(\mathring{Ric}^{3})-W_{ijkl}\mathring{R}_{ik}\mathring{R}_{jl}\right)dM_{g}.
\end{eqnarray*}
Thus, from Lemma~\ref{Keyestimte}, we have
\begin{eqnarray*}
\frac{n-2}{4(n-1)}\phi\mathcal{Y}(M,[g])&\leq&\frac{n-10}{4(n-1)}\int_{M}(1+f)^{2}R|\mathring{Ric}|^{2}dM_{g}\\
&&-\frac{3n-5}{2(n-1)}\int_{M}(1+f)^{2}|C_{ijk}|^{2}dM_{g}\\
&&+\sqrt{\frac{9(n-2)}{2(n-1)}}\int_{M}(1+f)^{2}\left(|W|^{2}+\frac{2n}{n-2}|\mathring{Ric}|^{2}\right)|\mathring{Ric}|^{2}dM_{g}.
\end{eqnarray*}
Next, from H\"older inequality, we obtain
\begin{eqnarray*}
\frac{n-2}{4(n-1)}\phi\mathcal{Y}(M,[g])&\leq&\frac{n-10}{4(n-1)}\int_{M}(1+f)^{2}R|\mathring{Ric}|^{2}dM_{g}\\
&&-\frac{3n-5}{2(n-1)}\int_{M}(1+f)^{2}|C_{ijk}|^{2}dM_{g}\\
&&+\sqrt{\frac{9(n-2)}{2(n-1)}}\phi\left(\int_{M}\left|W+\frac{\sqrt{n}}{\sqrt{2}(n-2)}\mathring{Ric}\varowedge g\right|^{\frac{n}{2}}dM_{g}\right)^{\frac{2}{n}}.
\end{eqnarray*}

Now, taking into account our integral pinching condition, see inequality (\ref{intP}),  we have that
$$\int_{M}(1+f)^{2}|C_{ijk}|^{2}dM_{g}\leq\frac{n-10}{6n-10}\int_{M}(1+f)^{2}R|\mathring{Ric}|^{2}dM_{g}.$$
Therefore, since we are assuming $4\leq n\leq10$ it is easy to see that, except to $n=10$, $M$ is a Einstein manifold and this concludes the proof of Theorem~\ref{CPEyamabe} in this case. Otherwise, for n=10, clearly we have $C_{ijk}=0$ and we are in position to use Theorem 1.2 in \cite{CHY14} (see also \cite{CHY16}) to conclude that $(M^{n}, g)$ is isometric to a standard sphere $\mathbb{S}^{n}.$
\end{proof}

\section{CPE metric satisfying $div^{4}W=0$}\label{div4W}

In the first part of this section we shall establish a integral identity for a CPE metric. Then we will able to prove a rigidity result for such metrics in dimension n, $n\geq4,$ provided that our manifold satisfies $div^{4}W=0.$ This result was motivated by the work of Qing and Yuan~\cite{QY} as well as by the very recent article of Santos~\cite{Alex}. In order to do this, it is worth reporting the following formula for the divergence of the Bach tensor (see \cite{CaoChen} for more details).

\begin{lemma}\label{lemaBACH}
Let $(M^{n},g),$ $n\geq3,$ be a Riemannian manifold. Then
\begin{equation*}
\nabla_{i}B_{ij}=\frac{n-4}{(n-2)^{2}}C_{jks}R_{ks}.
\end{equation*}
\end{lemma}

Now, we will show an integral formula that holds on every CPE metric.

\begin{proposition}\label{div3CCaux}
Let $(M^{n},g,f),$ $n\geq4,$ be a CPE metric. Then, for any $p\geq2,$ we have the following integral identity:
$$\frac{2}{p}\int_{M}(1+f)^{p}div^{3}CdM_{g}+\int_{M}(1+f)^{p}|C_{ijk}|^{2}dM_{g}=\frac{n-2}{n-1}\int_{M}(1+f)^{p-1}T_{ijk}C_{ijk}dM_{g},$$
where $T_{ijk}$ stand for the auxiliary tensor defined in (\ref{TensorT}).
\end{proposition}
\begin{proof}
Firstly, from Lemma~\ref{lemaBACH} and Eq. (\ref{CID}), we obtain
\begin{eqnarray}\label{divB1}
(n-2)\int_{M}(1+f)^{p}div^{2}BdM_{g}&=&\frac{n-4}{(n-2)}\int_{M}(1+f)^{p}\nabla_{i}(C_{ijk}R_{jk})dM_{g}\nonumber\\
&=&-\frac{p(n-4)}{(n-2)}\int_{M}(1+f)^{p-1}C_{ijk}\nabla_{i}fR_{jk}dM_{g}\nonumber\\
&=&\frac{p(n-4)}{2(n-1)}\int_{M}(1+f)^{p-1}C_{ijk}T_{ijk}dM_{g}.
\end{eqnarray}

On the other hand, from (\ref{bacha}), we achieve
\begin{eqnarray*}
(n-2)\int_{M}(1+f)^{p}div^{2}BdM_{g}&=&\int_{M}(1+f)^{p}\nabla_{j}\nabla_{i}\nabla_{k}C_{kij}dM_{g}\\
&&+\int_{M}(1+f)^{p}\nabla_{j}\nabla_{i}(W_{ikjl}R_{kl})dM_{g}\\
&=&\int_{M}(1+f)^{p}div^{3}CdM_{g}\\
&&-p\int_{M}(1+f)^{p-1}\nabla_{j}f\nabla_{i}(W_{ikjl}R_{kl})dM_{g}\\
&=&\int_{M}(1+f)^{p}div^{3}CdM_{g}-p\int_{M}(1+f)^{p-1}\nabla_{j}f\nabla_{i}W_{ikjl}R_{kl}dM_{g}\\
&&-p\int_{M}(1+f)^{p-1}\nabla_{j}fW_{ikjl}\nabla_{i}R_{kl}dM_{g}.
\end{eqnarray*}
So, it is not difficult to verify, using (\ref{cotton}) and (\ref{cottonwyel}), that the last expression can be rewritten as
\begin{eqnarray}\label{div3Caux}
(n-2)\int_{M}(1+f)^{p}div^{2}BdM_{g}&=&\int_{M}(1+f)^{p}div^{3}CdM_{g}\nonumber\\
&&-\frac{p(n-3)}{n-2}\int_{M}(1+f)^{p-1}C_{jlk}\nabla_{j}fR_{kl}dM_{g}\nonumber\\
&&-\frac{p}{2}\int_{M}(1+f)^{p-1}\nabla_{j}fW_{ikjl}C_{ikl}dM_{g}.
\end{eqnarray}
Now, the Equations (\ref{CID}) and (\ref{CTW}) substituted in (\ref{div3Caux}) allow us to conclude that
\begin{eqnarray}\label{divB2}
(n-2)\int_{M}(1+f)^{p}div^{2}BdM_{g}&=&\int_{M}(1+f)^{p}div^{3}CdM_{g}+\frac{p(n-3)}{2(n-1)}\int_{M}(1+f)^{p-1}C_{jlk}T_{jlk}dM_{g}\nonumber\\
&&+\frac{p}{2}\int_{M}(1+f)^{p-1}(fC_{ikl}-T_{ikl})C_{ikl}dM_{g}\nonumber\\
&=&\int_{M}(1+f)^{p}div^{3}CdM_{g}-\frac{p}{n-1}\int_{M}(1+f)^{p-1}C_{jlk}T_{jlk}dM_{g}\nonumber\\
&&+\frac{p}{2}\int_{M}(1+f)^{p}|C_{ijk}|^{2}dM_{g}.
\end{eqnarray}
Finally, comparing (\ref{divB1}) and (\ref{divB2}) we get the desired result.
\end{proof}

\begin{remark}
In three dimensional case, such integral formula was obtained by Qing and Yuan in \cite{QY}, see \cite[Corollary 4.1]{QY}.
\end{remark}

\subsection{Proof of Theorem~\ref{CPEdiv4W}}

\begin{proof}
Since we are considering $M$ satisfying zero radial Weyl curvature condition, it is immediate to see using Proposition~\ref{div3CCaux} and Eq. (\ref{CTW}) that,
\begin{equation}\label{div3CN}
\frac{2}{p}\int_{M}(1+f)^{p}div^{3}CdM_{g}+\frac{1}{n-1}\int_{M}(1+f)^{p}|C_{ijk}|^{2}dM_{g}=0.
\end{equation}
Hence, taking into account that our manifold has fourth order divergence-free Weyl tensor (i.e., ${\rm div}^{4}W=0$), we use (\ref{div3Cdiv4W}) and (\ref{div3CN}) jointly with Proposition~\ref{propLS} in order to deduce that the Cotton tensor vanishes identically. Therefore we may invoke Theorem 1.2 in \cite{CHY14} (see also \cite{CHY16}) to conclude that $(M^{n}, g)$ is isometric to a standard sphere $\mathbb{S}^{n}.$
\end{proof}

\begin{acknowledgement}
The author want to thank M. Matos Neto for helpful discussions about this subject.
\end{acknowledgement}

\end{document}